\documentclass[a4paper]{amsart}
\usepackage[12pt]{extsizes}
\usepackage{amsfonts,amsmath,amsthm,amscd,amssymb,latexsym,mathrsfs,enumitem}
\usepackage[english]{babel}
\usepackage{geometry}
\theoremstyle{plain}

\newtheorem{theorem}{Theorem}
\newtheorem{lemma}[theorem]{Lemma}
\newtheorem{proposition}[theorem]{Proposition}

\newtheorem{question}{Question}

\newtheorem{thmm}{Theorem}

\numberwithin{equation}{section}

\begin{document}

\title[A generalization of a Baire theorem]{A generalization of a Baire theorem concerning barely continuous functions}
\author{Olena Karlova${}^{1,2}$}

\address{${}^1$  Chernivtsi National University, Ukraine\\ ${}^2$ Jan Kochanowski University in Kielce, Poland}
\email{maslenizza.ua@gmail.com}


\begin{abstract}
  We prove that if $X$ is a paracompact space, $Y$ is a metric space and $f:X\to Y$ is a functionally fragmented map, then
(i) $f$ is $\sigma$-discrete and functionally $F_\sigma$-measurable; (ii)  $f$ is a Baire-one function, if $Y$ is weak adhesive and weak locally adhesive   for $X$; (iii) $f$ is countably functionally fragmented, if $X$ is Lindel\"{o}ff.

This result generalizes one   theorem of Rene Baire on classification of barely continuous functions.
\end{abstract}
\maketitle

\section{Introduction}

A map $f:X\to Y$ between topological spaces $X$ and $Y$ is said to be
\begin{itemize}
 \item[-] {\it  Baire-one}, if it is a pointwise limit of a sequence of continuous maps $f_n:X\to Y$;

 \item[-] {\it (functionally) $F_\sigma$-measurable} or {\it of the first (functional) Borel class}, if the preimage $f^{-1}(V)$ of any open set $V\subseteq Y$ is a union of a sequence of (functionally) closed sets in $X$;

 \item[-] \emph{barely continuous}, if the restriction $f|_F$ of $f$ to any non-empty closet set $F\subseteq X$ has a point of continuity.
\end{itemize}
Let us observe that the term ''barely continuous'' belongs to Stephens~\cite{steph}. However,  barely continuous
functions are also mentioned in literature as functions with the ''point of continuity property'' (see, for instance, \cite{Kum, Spurny}).

Among many other characterizations of Baire-one functions, the following classical Baire's theorem is well-known~\cite{Baire}.
\begin{thmm}
  For a complete metric space $X$ and a function $f:X\to\mathbb R$ the following conditions are equivalent:
\begin{enumerate}
  \item $f$ is Baire-one;

  \item $f$ is $F_\sigma$-measurable;

  \item $f$ is barely continuous.
\end{enumerate}
\end{thmm}

Recall that a map  $f:X\to Y$ between topological space   $X$ and a metric space  $Y$ is said to be   {\it fragmented}, if for all  $\varepsilon>0$ and nonempty closed set   $F\subseteq X$ there exists a relatively open set   $U\subseteq F$ such that
${\rm diam}f(U)<\varepsilon$. The above notion  was supposed by Jayne and Rogers~\cite{JR} in connection with Borel selectors of certain set-valued maps.

Evidently, every barely continuous map between a topological and a metric spaces is fragmented. Moreover, if $X$ is a hereditarily Baire space, then every fragmented function is barely continuous.  The property of baireness of $X$ is essential: let us consider a function   $f:\mathbb Q\to\mathbb R$, $f(r_n)=1/n$, where $\mathbb Q=\{r_n:n\in\mathbb N\}$ is the set of all rational numbers. Notice that $f$ is fragmented and everywhere discontinuous.

The next generalization of Baire's theorem follows from \cite[Corollary 7]{JOPV} and \cite[Theorem 2.1]{ACN}.

\begin{thmm}
  Let $X$ be a hereditarily Baire  paracompact perfect space,  $Y$ is a metric space and $f:X\to Y$. The following conditions are equivalent:
   \begin{enumerate}
     \item[(i)] $f$ is $F_\sigma$-measurable and $\sigma$-discrete;

     \item[(ii)] $f$ is fragmented.
   \end{enumerate}
   Moreover, if $Y$ is a convex subset of a Banach space, they are equivalent to:
    \begin{enumerate}
      \item[(iii)] $f$ is Baire-one.
    \end{enumerate}
\end{thmm}
Let us observe that a similar result for $Y=\mathbb R$ was obtained by Mykhaylyuk \cite{Mykhaylyuk:Visnyk:2000}.

The next theorem was recently  proved in \cite[Theorem 10]{Karlova:Mykhaylyuk:Comp}.
 \begin{thmm}
   If $X$ is a paracompact perfect   space, $Y$ is a metric contractible locally path-connected space and  $f:X\to Y$ is fragmented, then  $f\in{\rm B}_1(X,Y)$.
 \end{thmm}

The aim of this note is to extend the above mentioned results on maps defined on paracompact spaces which are not necessarily perfect (recall that a topological space is \emph{perfect} if every its open subset is $F_\sigma$).

The convenient tool of investigation of fragmented functions on non-perfect spaces is a concept of functional fragmentability introduced in~\cite{KM:arxiv1}. We prove a technical auxiliary result (Lemma~\ref{prop:sigmadd}) which connects regular families of functionally open sets in   paracompact spaces with the notion of $\sigma$-discrete decomposability. As an application of   this result we obtain (Theorem~\ref{th:bafs}) that for a paracompact space  $X$, a  metric space $Y$ and  a functionally fragmented map $f:X\to Y$ the following propositions hold: (i) $f$ is $\sigma$-discrete and functionally $F_\sigma$-measurable; (ii) $f$ is a Baire-one function, if $Y$ is weak adhesive and weak locally adhesive   for $X$; (iii) $f$ is countably functionally fragmented, if $X$ is Lindel\"{o}ff.

\section{Preliminaries}

Let $\mathscr U=(U_\xi:\xi\in[0,\alpha])$ be  a transfinite sequence of subsets of a topological space $X$. Then $\mathscr U$ is {\it regular in $X$}, if
\begin{enumerate}[label=(\alph*)]
  \item each $U_\xi$ is open in $X$;

  \item $U_0=\emptyset$, $U_\alpha=X$ and $U_\xi\subseteq U_\eta$ for all $0\le\xi\le\eta<\alpha$;

  \item\label{it:c} $U_\gamma=\bigcup_{\xi<\gamma} U_\xi$ for every limit ordinal $\gamma\in[0,\alpha]$.
\end{enumerate}

\begin{proposition}\label{prop:KM:arxiv}\cite[Proposition 1]{KM:arxiv}
  Let $X$ be a topological space, $(Y,d)$ be a metric space and $\varepsilon>0$. For a map $f:X\to Y$ the following conditions are equivalent:
  \begin{enumerate}
    \item\label{prop:char_frag:it:1}  $f$ is $\varepsilon$-fragmented;

    \item\label{prop:char_frag:it:2} there exists a regular sequence $\mathscr U=(U_\xi:\xi\in[0,\alpha])$ in $X$ such that ${\rm diam} f(U_{\xi+1}\setminus U_\xi)<\varepsilon$ for all $\xi\in[0,\alpha)$.
  \end{enumerate}
\end{proposition}

If a sequence $\mathscr U$ satisfies condition~(\ref{prop:char_frag:it:2}) of the previous proposition, then it is called {\it $\varepsilon$-associated with $f$} and is denoted by $\mathscr U_\varepsilon(f)$.

We say that an $\varepsilon$-fragmented  map $f:X\to Y$ is {\it functionally $\varepsilon$-fragmented} if $\mathscr U_\varepsilon(f)$ can be chosen such that every set $U_\xi$ is functionally open in $X$. Further, $f$ is {\it functionally fragmented}, if it is functionally $\varepsilon$-fragmented for each $\varepsilon>0$.

A map $f$ is {\it (functionally) countably fragmented}, if $f$ is (functionally) fragmented and every sequence $\mathscr U_\varepsilon$ can be chosen to be countable.

\section{A Lemma}
Let $\mathscr A$ be a family of subsets of a topological space $X$. Then $\mathscr A$ is called
\begin{itemize}
  \item \emph{discrete}, if each point $x\in X$ has a neighborhood which intersects at most one set from $\mathscr A$;

  \item  \emph{strongly functionally discrete} or, briefly, \emph{sfd-family}, if there exists a discrete family $(U_A:A\in\mathscr A)$ of functionally open subsets of $X$ such that $\overline{A}\subseteq U_A$ for every $A\in \mathscr A$.
\end{itemize}

Let us observe that every discrete family is strongly functionally discrete in collectionwise normal space.

\begin{lemma}\label{prop:sigmadd}{\rm (cf. \cite[Theorem 2]{ChZen})}
  Let $\mathscr U=(U_\xi:\xi\in[0,\alpha])$ be a regular  family   of  functionally open sets in a paracompact space $X$. Then there exists a sequence $(\mathscr F_n)_{n\in\omega}$ of families $\mathscr F_n=(F_{\xi,n}:\xi\in[0,\alpha])$ such that
  \begin{enumerate}
\item    $U_\xi\setminus \bigcup_{\eta<\xi}U_\eta=\bigcup_{n\in\omega}F_{\xi,n}$ for all $\xi\in[0,\alpha)$,\label{lem:sigmadd:it1}

\item   $\mathscr F_n$ is an sfd-family in $X$ for all $n\in\omega$,\label{lem:sigmadd:it2}

\item $F_{\xi,n}$ is closed in $X$ for all $n\in\omega$ and $\xi\in[0,\alpha)$.\label{lem:sigmadd:it3}
  \end{enumerate}
\end{lemma}

\begin{proof}
  For every $\xi\in[1,\alpha]$ we denote $P_\xi=U_\xi\setminus \bigcup_{\eta<\xi}U_\eta$. Since every $P_\xi$ is functionally $G_\delta$ in $X$ as a difference of two functionally open sets, we can choose a sequence $(G_{\xi,n})_{n\in\omega}$ of functionally open sets such that
  \begin{gather*}
  P_\xi=\bigcap_{n\in\omega} G_{\xi,n}\quad\mbox{for all}\,\,\, \xi\in[1,\alpha)\quad\mbox{and}\\
   G_{\xi,n}\subseteq U_{\xi} \quad\mbox{for all}\,\,\, \xi\in[1,\alpha), n\in\omega.
  \end{gather*}

We put
$$
I=\bigcup_{k\in\omega}\omega ^k
$$
and define by the induction on $k$ sequences $(\mathscr U_i:i\in I)$ and $(\mathscr V_i:i\in I)$ of open coverings of $X$ such that
\begin{enumerate}
  \item[(a)] $\mathscr U_{\emptyset}=\mathscr U$;

  \item[(b)]  $\mathscr V_i$ is a locally finite barycentric refinement of $\mathscr U_i$  for all $i\in \omega^k$;

  \item[(c)] for all $i\in\omega^k$ and $n\in\omega$ we have $\mathscr U_{(i,n)}=(U_{\xi,(i,n)}:\xi\in[0,\alpha))$, where
  $$
  C_{\xi,i}=\overline{\{x\in X:{\rm St}(x,\mathscr V_i)\subseteq \bigcup_{\eta<\xi}U_\eta\}}\quad \mbox{and}\quad U_{\xi,(i,n)}= G_{\xi,n} \setminus C_{\xi,i}
  $$
\end{enumerate}
for all $k\in\omega$. Let us observe that the existence of families $\mathscr V_i$ follows from the paracompactness of $X$ (see \cite[Theorem 5.1.12]{Eng-eng}).

Notice that
$$
C_{\xi,i} \subseteq \bigcup_{\eta<\xi}U_\eta,
$$
because $\mathscr V_i$ is an open covering of $X$. Therefore, since $(P_\xi:\xi\in[0,\alpha])$ is a partition of $X$, $\mathscr U_{(i,n)}$ defined in (c) covers $X$ for all $n\in\omega$.

For every $x\in X$ we put
  $$
  \mu(x)=\min\{\xi\in[0,\alpha):x\in U_\xi\}
  $$
  and show that
  \begin{gather}\label{gat:ttt1}
\forall x\in X\,\,\,\exists i\in I: \,\,\,  {\rm St}(x,\mathscr V_i)\subseteq U_{\mu(x)}.
  \end{gather}
Assume to the contrary that there exists $x\in X$ such that (\ref{gat:ttt1}) is not true. Since each family $\mathscr V_i$ is locally finite refinement of $\mathscr U$, for every $i\in I$ there is $\xi_i$ such that ${\rm St}(x,\mathscr V_i)\subseteq U_{\xi_i}$. Let $\xi(x)=\min\{\xi_i:i\in I\}$. Then $\xi(x)>\mu(x)$. Therefore, $x\not\in P_{\xi(x)}$ and we can take $j\in\omega$ such that $x\not\in G_{\xi(x),j}$.

From the definition of the sequence $\mathscr U_{(i,j)}$ it follows that $x\not\in U_{\xi(x),(i,j)}$. Since ${\rm St}(x,\mathscr V_i)\subseteq U_{\xi(x)}$, we have  $x\not\in U_{\xi,(i,j)}$ for all $\xi>\xi(x)$. Therefore,
\begin{gather}\label{gath:f}
x\not\in \bigcup_{\xi\ge\xi(x)}U_{\xi,(i,j)}
\end{gather}

By (b) there exists $\beta\in[0,\alpha)$ such that ${\rm St}(x,\mathscr V_{(i,j)})\subseteq U_{\beta,(i,j)}$. It follows from~(\ref{gath:f})  that $\beta<\xi(x)$. The inclusion $U_{\beta,(i,j)}\subseteq U_\beta$ contradicts to the choice of $\xi(x)$.

  Let $(\mathscr V_i:i\in I)=(\mathscr H_n:n\in\omega)$. Now for all $\xi\in[0,\alpha)$ and $n\in\omega$ we put
  \begin{gather*}
  D_{\xi,n}=\{x\in P_\xi: {\rm St}(x,\mathscr H_n)\subseteq U_\xi\}\quad\mbox{and}\quad   F_{\xi,n}=\overline{D_{\xi,n}}.
  \end{gather*}
We will show that
$$
P_\xi=\bigcup_{n\in\omega} F_{\xi,n}
$$
for all $\xi\in[0,\alpha)$. Property~(\ref{gat:ttt1}) implies that $P_{\xi}\subseteq \bigcup_{n\in\omega} F_{\xi,n}$. Now assume that $x\in F_{\xi,n}$ for some $\xi$ and $n$. Put $O={\rm St}(x,\mathscr H_n)\cap U_{\mu(x)}$. Then $O\cap D_{\xi,n}\ne\emptyset$. Take any $y\in O$. Since $y\in U_{\mu(x)}$ and $y\in P_\xi$, $\mu(x)\ge \xi$. The inclusions ${\rm St}(y,\mathscr H_n)\subseteq U_\xi$ and $y\in{\rm St}(x,\mathscr H_n)$ imply that $x\in U_\xi$. Hence, $\mu(x)\le\xi$. Therefore, $\mu(x)=\xi$. Then $x\in P_{\xi}$. Moreover, it follows that the family $\mathscr F_n=(F_{\xi,n}:\xi\in[0,\alpha))$ is discrete in $X$.

Since $X$ is paracompact, $X$ is collectionwise normal, which implies that $\mathscr F_n$ is strongly functionally discrete family for all $n\in\omega$.
 \end{proof}

 \section{An application of Lemma to classification  of fragmented functions}
Let $X$ be a topological space. Recall that a topological space $Y$ is
 \begin{itemize}
 \item \emph{an adhesive for $X$}, if for any disjoint functionally closed sets $A$ and  $B$ in $X$ and for any two continuous maps $f,g:X\to Y$ there exists a continuous map $h:X\to Y$ such that $h|_A=f|_A$ and $h|_B=g|_B$;

     \item \emph{\it a weak adhesive for $X$}, if for any two points   $y,z\in Y$ and  disjoint functionally closed sets   $A$ and $B$ in $X$ there exists a continuous map   $h:X\to Y$ such that $h|_A=y$ i $h|_B=z$;

   \item \emph{\it a locally weak adhesive for $X$}, if for every   $y\in Y$ and every neighborhood   $V\subseteq Y$  of $y$ there exists a neighborhood   $U$ of  $y$ such that $U\subseteq V$ and for every  $z\in U$ there exists a continuous map  $h:X\to V$ with $h|_A=y$ and $h|_B=z$.
 \end{itemize}
 It was proved in  \cite[Theorem 2.7]{Karlova:Mykhaylyuk:Stable} that any topological space   $Y$  is an adhesive for every strongly zero dimensional space $X$; a path-connected space  $Y$ is an adhesive for any compact space $X$ each point of which has a base of neighborhoods with discrete boundaries;  $Y$ is an adhesive for any space   $X$ if and only if   $Y$ is contractible. Moreover, it is easy to see that every (locally) path-connected space is a (locally) weak adhesive for any $X$.

A family $\mathscr B$ of subsets of a topological space $X$ is said to be {\it a base} for a map $f:X\to Y$, if for every open set $V\subseteq Y$ there exists a subfamily $\mathscr B_V$ of $\mathscr B$ such that $f^{-1}(V)=\bigcup_{B\in\mathscr B_V}B$.

A map $f:X\to Y$ is \emph{$\sigma$-discrete}, if there is a sequence $(\mathscr B_n)_{n\in\omega}$ of discrete families of sets in $X$ such that the family $\bigcup_{n\in\omega}\mathscr B_n$ is a base for $f$.

\begin{theorem}\label{th:bafs}
  Let $X$ be a paracompact space, $Y$ be a metric space and $f:X\to Y$ be a functionally fragmented map. Then
   \begin{enumerate}
     \item $f$ is $\sigma$-discrete and functionally $F_\sigma$-measurable;

     \item $f$ is a Baire-one function, if $Y$ is weak adhesive and weak locally adhesive   for $X$;

     \item $f$ is countably functionally fragmented, if $X$ is Lindel\"{o}ff.
   \end{enumerate}
\end{theorem}

\begin{proof}
\textbf{1)} For every $n\in\mathbb N$ we choose a  family $\mathscr U_{1/n}(f)=(U_{\xi,n}:\xi\in[0,\alpha_n])$ consisting of functionally open sets $U_{\xi,n}$.
We claim that the family $\mathscr P=\bigcup_{n\in\mathbb N}\mathscr P_n$ is a base for $f$, where $\mathscr P_n=(U_{\xi,n}\setminus\bigcup_{\eta<\xi}U_{\eta,n}:\xi\in[0,\alpha_n])$, $n\in\mathbb N$. Indeed, fix an open set $V$ in $Y$ and take any $x\in f^{-1}(V)$. Find $n\in\mathbb N$ such that an open ball $B$ with the center at $f(x)$ and radius $1/n$ contains in $V$. Since $\mathscr P_n$ is a partition of $X$, there exists $\xi\in[0,\alpha_n]$ such that $x\in P_{\xi,n}=U_{\xi,n}\setminus\bigcup_{\eta<\xi}U_{\eta,n}$. Evidently, $f(P_{\xi,n})\subseteq B\subseteq V$.

By Lemma~\ref{prop:sigmadd} for every $n\in\mathbb N$ there exists a sequence $(\mathscr F_{n,k})_{k\in\omega}$ of families $\mathscr F_{n,k}=(F_{\xi,n,k}:\xi\in[0,\alpha_n])$ which satisfies conditions (\ref{lem:sigmadd:it1})--(\ref{lem:sigmadd:it3}) of Lemma~\ref{prop:sigmadd}.  Properties (\ref{lem:sigmadd:it1}) and (\ref{lem:sigmadd:it2}) imply that the family $\mathscr B=\bigcup\limits_{k,n} \mathscr F_{n,k}$ is a $\sigma$-discrete base for $f$ consisting of closed sets. It follows that $f$ is $F_\sigma$-measurable and  a $\sigma$-discrete map.
 Finally, \cite[Proposition 2.6 (iv)]{K:EJM:2016} implies that $f$ is functionally $F_\sigma$-measurable.

Property \textbf{2)} follows from 1) and \cite[Theorem 3.2]{K:PIGC}.

   \textbf{3)} It is enough to show that every regular sequence consisting of functionally open sets in a Lindel\"{o}ff space $X$ is countable.

   Let  $\mathscr U=(U_\xi:\xi\in[0,\alpha])$ be a regular covering of $X$ by functionally open sets $U_\xi$. There exists a sequence $(\mathscr F_n)_{n\in\omega}$  of   families   in $X$ such that conditions (\ref{lem:sigmadd:it1})--(\ref{lem:sigmadd:it3}) of Lemma~\ref{prop:sigmadd} are valid. Notice that every family $\mathscr F_n$ is at most countable, since it is discrete and $X$ is Lindel\"{o}ff. We consider an enumeration $\{F_k:k\in\omega\}$ of the family $\bigcup_{n\in\omega}\mathscr F_n$. Let $\varphi:[0,\alpha)\to 2^\omega$ be a map,
 $$
 \varphi(\xi)=\{k\in\omega: F_k\subseteq P_\xi\}.
 $$
Since $(\varphi(\xi):\xi\in[0,\omega_1))$ is a family of mutually disjoint subsets of $\omega$, it is at most countable.
\end{proof}

We do not know the answer to the following question.
\begin{question}
  Is it true that every fragmented Baire-one real-valued function defined on a paracompact Hausdorff space is functionally fragmented?
\end{question}

\end{document}